\documentclass[12pt]{article}
\usepackage[T1]{fontenc}
\usepackage[utf8]{inputenc}
\usepackage[english]{babel}
\usepackage[stretch=10]{microtype}
\usepackage{mathtools}
\usepackage{amsthm}
\usepackage{amssymb}
\usepackage{dsfont}
\usepackage{extarrows}
\usepackage{enumerate}
\usepackage{cleveref}
\usepackage{fullpage}
\usepackage{framed}
\usepackage{color}

\theoremstyle{plain} \newtheorem{theorem}{Theorem}[section]
\theoremstyle{plain} \newtheorem{proposition}[theorem]{Proposition}
\theoremstyle{plain} \newtheorem{lemma}[theorem]{Lemma}
\theoremstyle{plain} 
\theoremstyle{definition} 
\theoremstyle{definition} 
\theoremstyle{remark} 
\theoremstyle{remark} 

\renewcommand{\P}{\mathbf P}
\newcommand{\R}{\mathbb R}
\renewcommand{\SS}{\mathbb S}

\newcommand{\dint}{\mathrm{d}}

\DeclareMathOperator{\Var}{\mathbf{Var}}
\DeclareMathOperator{\E}{\mathbf E}
\DeclareMathOperator{\Vis}{Vis}
\DeclareMathOperator{\conv}{conv}

\DeclareMathOperator{\aff}{aff}
\DeclareMathOperator{\vol}{vol}
\DeclareMathOperator{\proj}{proj}

\newcommand{\indicator}{\boldsymbol{1}}

\usepackage{fancybox}
\newlength{\querylen}
\allowdisplaybreaks[3]
\begin{document}
\title{\Large \textbf{Limit theorems for random polytopes\\ with vertices on convex surfaces}}

\author{Nicola Turchi\footnotemark[1]\;\footnotemark[2]\; and Florian Wespi\footnotemark[3]} 
\date{}

\renewcommand{\thefootnote}{\fnsymbol{footnote}}
\footnotetext[1]{Faculty of Mathematics, Ruhr University Bochum, Germany. E-mail: nicola.turchi@rub.de}

\footnotetext[3]{Institute of Mathematical Statistics and Actuarial Science, University of Bern, Switzerland. E-mail: florian.wespi@stat.unibe.ch}

\footnotetext[2]{Research supported by the Deutsche Forschungsgemeinschaft (DFG) via RTG2131 \textit{High-dimensional Phenomena in Probability -- Fluctuations and Discontinuity}.}

\maketitle

\begin{abstract}
\noindent The random polytope \(K_n\), defined as the convex hull of \(n\) points chosen uniformly at random on the boundary of a smooth convex body, is considered. Proofs for lower and upper variance bounds, strong laws of large numbers and central limit theorems for the intrinsic volumes of \(K_n\) are presented. A normal approximation bound from Stein's method and estimates for surface bodies are among the involved tools.
\bigskip

\noindent \textbf{Keywords}. Central limit theorem, intrinsic volume, random polytope,  stochastic geometry, surface body, variance.

\noindent \textbf{MSC (2010)}. 52A22, 60D05, 60F05.
\end{abstract}

\section{Introduction and main results}

For fixed \(d\ge 2\), let \(\mathcal{K}^2_+\) be the set of convex bodies in \(\R^d\) which have a twice differentiable boundary with everywhere positive Gaussian curvature. Let \(K\in\mathcal{K}^2_+\) be a convex body. We denote by $\mathcal{H}^{d-1}$ the $(d-1)$-dimensional Hausdorff measure on $\partial K$, normalized such that $\mathcal{H}^{d-1}(\partial K)=1$. For \(n\ge d+1\), we choose random points \(X_1,\ldots,X_n\) from \(\partial K\), independently and according to \(\mathcal{H}^{d-1}\). We denote by \(K_n\) the convex hull of \(X_1,\ldots, X_n\). This means that \(K_n\) is a random polytope having its vertices on the boundary of \(K\). The interest of this paper is about the intrinsic volumes \(V_\ell(K_n)\) of \(K_n\), \(\ell\in\{1,\ldots,d\}\). The importance of these functionals is well-known and arises from convex and integral geometry. Indeed, as Hadwiger's theorem states, they form (together with the Euler characteristic) a basis of the vector space of all motion invariant and continuous valuations on convex bodies. With this paper we provide lower and upper variance bounds, strong laws of large numbers and central limit theorems for \(V_\ell(K_n),\,\ell\in\{1,\ldots,d\}\), filling some gaps that remain in the study of these objects.

\bigskip

Intrinsic volumes have been studied extensively in the alternative setting of random polytopes that arise as convex hulls of points chosen uniformly at random inside a fixed convex body. Results concerning the expectation of \(V_\ell(K_n)\),
\(\ell\in\{1,\ldots,d\}\), have been studied, for example, by Reitzner \cite{Reitzner2004}, variance bounds can be found in B\"or\"oczky, Fodor, Reitzner and V\'igh \cite{BFRV2009} and B\'ar\'any, Fodor and V\'igh \cite{Barany2010}, and central limit theorems were treated in Reitzner \cite{Reitzner2005}, Vu \cite{Vu2006},  Lachi\`eze-Rey, Schulte and Yukich \cite{Lachieze-Rey2017} and Th\"ale, Turchi and Wespi \cite{TTW2017}. More details can be found in the references therein.

\bigskip

On the other hand, the approximation of a convex body \(K\), by means of a sequence of random polytopes \(K_n\), is improved whenever the vertices of \(K_n\) are restricted to lie on the boundary of \(K\), therefore making it a model worth studying. Indeed, in this framework the expectations of \(V_\ell(K_n)\), \(\ell\in\{1,\ldots,d\}\), have been studied, for example, by Buchta, M\"uller and Tichy \cite{Buchta1985}, Reitzner \cite{Reitzner2002}, Sch\"utt and Werner \cite{Schutt2003} and B\"or\"oczky, Fodor and Hug \cite{BFH2013}. However, more detailed informations are only known about the distribution of the volume \(V_d(K_n)\). In particular, an upper variance bound was found by Reitzner \cite{Reitzner2003} and a lower variance bound together with concentration inequalities by Richardson, Vu and Wu \cite{Richardson2008}. Only recently, Th\"ale \cite{Thale2017} obtained a quantitative central limit theorem for \(V_d(K_n)\) based on Stein's method. 

\bigskip

Our first aim is to generalize the results obtained in \cite{Reitzner2003, Richardson2008} to \(V_\ell(K_n)\), \(\ell\in\{1,\ldots,d\}\). In fact, we prove lower variance bounds following the ideas of \cite{Barany2010,Reitzner2005, Richardson2008} and upper variance bounds in the manner of \cite{Barany2010}, making use of the Efron-Stein jackknife inequality from \cite{Reitzner2003}. In particular, the upper variance bounds imply strong laws of large numbers as in \cite{Barany2010}. Secondly, we prove quantitative central limit theorems for \(V_\ell(K_n)\), \(\ell\in\{1,\ldots,d\}\), using a normal approximation bound obtained in \cite{LachiezeRey}, extending the result of \cite{Thale2017}.

\bigskip

We now introduce some notation in order to present our results. Let $(a_n)_{n\in\mathbb{N}}$ and $(b_n)_{n\in\mathbb{N}}$ be two sequences of real numbers. We write $a_n\ll b_n$ (or $a_n\gg b_n$) if there exist a constant $c\in(0,\infty)$ and a positive number $n_0$ such that $a_n\leq c \,b_n$ (or $a_n\geq c \,b_n$) for all $n\geq n_0$. Furthermore, $a_n=\Theta(b_n)$ means that $b_n\ll a_n \ll b_n$.

\bigskip

Our first result concerns asymptotic lower and upper bounds, respectively, for the variances of the intrinsic volumes.

\begin{theorem}
%\label{thm:lowerbound}
\label{thm:variancebounds}
	Let $K\in\mathcal{K}_{+}^2$ and choose $n$ random points on $\partial K$ independently 
	and according to the probability distribution $\mathcal{H}^{d-1}$. Then, 
	\begin{equation*}
	\Var [V_{\ell}(K_n)]=\Theta\bigl(n^{-\frac{d+3}{d-1}}\bigr), \quad \ell\in\{1,\dots,d\}.
	\end{equation*}
\end{theorem}
Based on a result stated in \cite[Theorem 1]{Reitzner2002} concerning the behaviour of \(V_\ell(K)-\E[V_\ell(K_n)]\), 
the upper variance bounds of \Cref{thm:variancebounds} imply strong laws of large numbers.
\begin{theorem}
	\label{thm:strongnumbers}
	In the set-up of \Cref{thm:variancebounds}, it holds that
\begin{equation*}
\P\Bigl(\lim_{n\to\infty}\bigl(V_\ell(K)-V_\ell(K_n)\bigr)\cdot n^{\frac{2}{d-1}}=c_{d,\ell,K}
%\int_{\partial K} H_{d-1}(x)^{\frac{1}{d-1}}{H}_{d-\ell}(x)\,\mathcal{H}^{d-1}(\dint x)
\Bigr)=1,\quad \ell\in\{1,\dots,d\}, 
\end{equation*}
for some positive constants \(c_{d,\ell,K}\) that depend on \(d,\ell\) and \(K\).
\end{theorem}

The constants \(c_{d,\ell,K}\) appear in an explicit form in \cite[Theorem 1]{Reitzner2002} and can be expressed in form of integrals of the principal curvatures of \(K\).

\bigskip

Next, we introduce the standardized intrinsic volume functionals, defined by
\begin{equation*}
W_{\ell}(K_n):=\frac{V_{\ell}(K_n)-\E[V_{\ell}(K_n)]}{\sqrt{\Var[V_{\ell}(K_n)]}}, \quad \ell\in\{1,\dots,d\}.
\end{equation*}

We prove the following central limit theorems for such functionals.
	
\begin{theorem}
	\label{thm:CLT}
In the set-up of \Cref{thm:variancebounds}, it holds that
	\begin{equation*}
\sup_{u\in\R}|\P(W_\ell(K_n)\leq u)-\P(N\leq u)|	\ll  n^{-{\frac{1}{2}}}(\log{n})^{3+{\frac{6}{d-1}}}, \quad \ell\in\{1,\dots,d\},
	\end{equation*}
	where \(N\) is a standard Gaussian random variable. In particular, \(W_{\ell}(K_n)\) converges in distribution to \(N\), as \(n\to\infty\).
\end{theorem}

Note that the rate of convergence in \Cref{thm:CLT} does not depend on \(\ell\). Moreover, the same rate of convergence was already obtained in \cite{Thale2017} for the case \(\ell=d\).

\bigskip

The paper is organised as follows. In \Cref{Sec:Background} we introduce the notation and recall some background material from convex geometry, results concerning the surface and floating bodies and the normal bound from \cite{LachiezeRey} that will be used in the proof of \Cref{thm:CLT}. In \Cref{Sec:lowerbound} we present the geometric construction needed for the proof of the lower bounds of \Cref{thm:variancebounds} and the proof itself. In \Cref{Sec:upperbounds} we prove the upper bounds of \Cref{thm:variancebounds} by means of the Efron-Stein jackknife inequality and we also prove \Cref{thm:strongnumbers}, which directly follows from the former. Finally, in \Cref{Sec:centrallimitheorems} we give the proof of \Cref{thm:CLT}.

%Fix a space dimension $d\geq2$. Let $\mathcal{K}_{+}^2$ be the space of convex bodies $K\subset\R^d$ (non-empty compact convex sets) which have a boundary $\partial K$ that is twice differentiable and has 
%strictly positive Gaussian curvature at every point. We denote by $\mathcal{H}^{d-1}$ the 
%$(d-1)$-dimensional Hausdorff measure on $\partial K$ and we normalize it in such a 
%way that $\mathcal{H}^{d-1}(\partial K)=1$. We choose $n$ random points on $\partial K$
%independently according to $\mathcal{H}^{d-1}$. These points are denoted by $X_1,\dots,X_n$. 
%For a fixed integer $n\geq d+1$, the random convex hull
%\begin{equation*}
%K_n:=\conv\{X_1,\dots,X_n\}
%\end{equation*}
%is a random polytope in $K$ with all its vertices on the boundary of $K$. 
%\begin{framed}
%\dots
%\end{framed}

\section{Background material}\label{Sec:Background}

\subsection{General notation}

The closed Euclidean ball of radius $r$ centred at $x\in\R^d$ is denoted by $B^d(x,r)$, and $B^d=B^d(0,1)$ stands for the centred Euclidean unit ball. The boundary of \(B^d\) is indicated with \(\SS^{d-1}\). Moreover, the volume of $B^d$ is denoted by $\kappa_d=\pi^{d/2}\Gamma(1+{\frac{d}{2}})^{-1}$. For a finite set $A=\{x_1,\ldots,x_n\}\subset \R^d$, the convex hull of $A$ is denoted by $[x_1,\ldots,x_n]$. The vectors $e_1,\dots,e_d$ represent the standard orthonormal basis of $\R^d$. We indicate with \(\sphericalangle(u,v) \) the angle between two vectors \(u,v\) \(\in\R^d\). For a linear subspace \(V\) of \(\R^d\), we define \( \sphericalangle(u,V)\coloneqq \inf\{\sphericalangle(u,v): v\in V\}\). Given a subset \(U\in\R^d,\) its projection onto \(\R^{d-1}\) is denoted by \(\mathrm{proj}_{\R^{d-1}}U=\{x\in\R^{d-1}: (x,y)\in U\text{ for some } y\in\R\}\). For a function \(f\colon \R^d\to\R\), we say \(f\in\mathcal{C}^2\) if it is twice differentiable with continuous second order partial derivatives.
\bigskip 

Let $u\in\R^d$ and $h\in\R$. We denote by $H(u,h)$ the hyperplane $\{x\in\R^d : \langle x,u\rangle=h\}$. The corresponding halfspace $\{x\in\R^d : \langle x,u\rangle\geq h\}$ is denoted by
$H^+(u,h)$. Often one describes a convex body by its support function. The support function of $K$ is defined by
\begin{equation*}
h_K(u)=\sup\{\langle x,u\rangle : x\in K\}, \quad u\in\SS^{d-1}.
\end{equation*}
Since $K\in\mathcal{K}_{+}^2$, there exists a unique unit outward normal $u_x$ for each $x\in\partial K$.
The intersection of $K$ with $H^+(u_x,h_K(u_x)-h)$ is denoted by
$C^K(x,h)$. We call $C^K(x,h)$ a cap of $K$ at $x$ of height $h$. A cap $C^K$ is called an
$\varepsilon$-cap if $V_d(C^K)=\varepsilon$, where $V_d(\cdot)$ denotes the $d$-dimensional 
volume. Analogously, a cap $C^K$ with 
$\mathcal{H}^{d-1}(C^K\cap \partial K)=\varepsilon$ is called an $\varepsilon$-boundary cap.

\bigskip 

Let $\ell\in\{0,\dots,d\}$. We denote by $G(d,\ell)$ the Grassmannian of all $\ell$-dimensional linear subspaces of $\R^d$, which is supplied with the unique Haar probability measure $\nu_{\ell}$, see \cite{Schneider2014}. For $L\in G(d,\ell)$, we write $\vol_{\ell}(K|L)$ to indicate the $\ell$-dimensional Lebesgue measure of the orthogonal projection of $K$ onto $L$. Then, the $\ell\text{-th}$ intrinsic volume of a convex body $K$ can be defined as
\begin{equation}\label{eq.Kubota}
V_{\ell}(K) := \binom{d}{\ell} \frac{\kappa_d}{\kappa_{\ell}\kappa_{d-\ell}}\int_{G(d,\ell)}\!\!\!\vol_{\ell}(K|L)\,\nu_{\ell}(\dint L)\,.
\end{equation}
In particular, $V_d(K)$ is the ordinary volume (Lebesgue measure), $V_{d-1}(K)$ is half of the surface area, $V_1(K)$ is a constant multiple of the mean width and $V_0(K)$ is the Euler-characteristic of $K$. 

\bigskip

We define the function $v : K\to\R$ by
\begin{equation*}
  v(x)\coloneqq\min\{V_d(K\cap H) : H\ \text{is a half space in}\ \R^d \ \text{containing}\ x\}.
\end{equation*}
Then, the set
\begin{equation*}
  K(v\geq t)\coloneqq\{x\in K : v(x)\geq t\}
\end{equation*}
is called the floating body of $K$ with parameter $t>0$. The wet part of $K$ is defined by
\begin{equation*}
  K(t)=K(v\leq t)\coloneqq\{x\in K : v(x)\leq t\}.
\end{equation*}
In a similar way, we define the function $s: K\to\R$ by
\begin{equation*}
s(x):=\min\{\mathcal{H}^{d-1}(\partial K\cap H): H\ \text{is a half space in}\ \R^d \
\text{containing}\ x\}.
\end{equation*}
The surface body of $K$ with parameter $t>0$ is defined by
\begin{equation*}
K(s\geq t):=\{x\in K : s(x)\geq t\}.
\end{equation*}
Analogously, we set
\begin{equation*}
K(s\leq t):=\{x\in K : s(x)\leq t\}.
\end{equation*}
We define the visibility region (with respect to \(s\)) of a point \(z\in\partial K\) with parameter \(t>0\) as
\begin{equation*}
	\Vis_z(t) := \{x\in K(s\leq t) :[x,z]\cap K(s\geq t)=\emptyset\},
\end{equation*}
where \([x,z]\) denotes the closed line segment which connects \(x\) and \(z\).

\bigskip

We use the convention that constants with the same subscript may differ from section to section.

\subsection{Geometric tools}
\label{sec:geometric}
%We will use some methods which were already applied in 
%\cite{Barany2010, Barany2016, Reitzner2002, Reitzner2005, Richardson2008}. 
The concept of the surface body is convenient in view of \Cref{lemma:surface-body}, which clarifies its connection with the random polytope \(K_n\).

\begin{lemma}
	\label{lemma:surface-body}
	
	\textup{\cite[Lemma~4.2]{Richardson2008}} For all $\alpha\in(0,\infty)$, there 
	exists a constant $c_{\alpha}\in(0,\infty)$ only depending on $\alpha$ such that
\begin{equation*}
	\P(K(s\geq\tau_n)\not\subseteq K_n)\leq n^{-\alpha},
\end{equation*}
where
\begin{equation*}
\tau_n:=c_{\alpha}\frac{\log{n}}{n}.
\end{equation*}
\end{lemma}
In the following, we present some well-known geometric results in order to keep our presentation reasonably self-contained. For every point $x\in\partial K$, there exists a paraboloid $Q_x$, given by a quadratic form $b_{Q_x}$, osculating at $x$.
The following precise description of the local behaviour of the boundary 
of a convex body $K\in\mathcal{K}_{+}^2$ is due to Reitzner \cite{Reitzner2002}.

\begin{lemma}\textup{\cite[Lemma 6]{Reitzner2002}}\label{lemma:approx}
	Let $K\in\mathcal{K}_{+}^2$  and choose $\delta> 0$ sufficiently small. 
	Then, there exists a $\lambda> 0$, only depending on $\delta$ and $K$, such 
	that for each $x\in\partial K$ the following holds. Identify the hyperplane 
	tangent to $K$ at $x$ with $\R^{d-1}$ and $x$ with the origin. The $\lambda$-neighbourhood 
	$U^{\lambda}$ of $x$ in $\partial K$ defined by $\proj_{\R^{d-1}}U^{\lambda}=\lambda B^{d-1}$ 
	can be represented by a convex function $f^{(x)}(y)\in\mathcal{C}^2$, i.e., 
	$(y,f^{(x)}(y))\in\partial K$ for $y\in\lambda B^{d-1}$. Denote by $f_{ij}^{(x)}(0)$ the second partial derivatives of $f^{(x)}$ at the origin. Then, 
	\begin{equation*}
	b_{Q_x}(y)=\frac{1}{2}\sum_{i,j}f_{ij}^{(x)}(0)y_iy_j
	\end{equation*}
	and it holds that
	\begin{equation*}
	(1+\delta)^{-1}b_{Q_x}(y)\leq f^{(x)}(y)\leq (1+\delta)\,b_{Q_x}(y)
	\end{equation*}
	for $y\in\lambda B^{d-1}$.
\end{lemma}

%Note that $b_{Q_x}(y)$ can be written as
%\begin{equation*}
%b_{Q_x}(y)=\frac{1}{2}(\kappa_1(x)\langle y,e_1\rangle^2+\cdots
%+\kappa_{d-1}(x)\langle y,e_{d-1}\rangle^2),
%\end{equation*}
%where $\{e_1,\dots,e_{d-1}\}$ is a suitable Cartesian coordinate system in $\R^{d-1}$ 
%and $\kappa_{1}(x),\dots,\kappa_{d-1}(x)$ are the principal curvatures of $\partial K$ 
%at $x$, see \cite{Reitzner2002}. 

In the next Lemma we state two well-known relations regarding $\varepsilon$-caps and $\varepsilon$-boundary caps. 
\begin{lemma}\textup{\cite[Lemma 6.2]{Richardson2008}}\label{relation-of-caps}
	For a given $K\in\mathcal{K}_{+}^2$, there exist constants $\varepsilon_0, c_1, c_2>0$
	such that for all $0<\varepsilon<\varepsilon_0$ we have that for any $\varepsilon$-cap $C^K$ of
	$K$,
	\begin{equation*}
	c_1^{-1}\varepsilon^{(d-1)/(d+1)}\leq\mathcal{H}^{d-1}(C^K \cap \partial K)\leq c_1\varepsilon^{(d-1)/(d+1)}
	\end{equation*}
	and for any $\varepsilon$-boundary cap $\widetilde{C}^K$ of K,
	\begin{equation*}
	c_2^{-1}\varepsilon^{(d+1)/(d-1)}\leq V_d(\widetilde{C}^K)\leq c_2\varepsilon^{(d+1)/(d-1)}.
	\end{equation*}
\end{lemma}

For the next geometrical Lemma we assume that $\varepsilon$ is sufficiently small.
\begin{lemma}\label{lemma-Vu}\textup{\cite[Lemma 6.2]{Vu2005}}
Let $x$ be a point on the boundary of $K$ and $D(x,\varepsilon)$ the set of all points on the 
boundary which are of distance at most $\varepsilon$ to $x$. Then, the convex hull of $D(x,\varepsilon)$ 
has volume at most $c_3\varepsilon^{d+1}$, where $c_3>0$ is a constant.
\end{lemma}

The following result is known as the economic cap covering theorem, see \cite{Barany2010, Barany1988}.
\begin{proposition}\label{economic-cap-theorem}\textup{\cite[Theorem 4]{Barany2010}}
	Assume that $K$ is a convex body with unit volume and let $0<t<t_0=(2d)^{-2d}$. 
	Then, there are caps $C_1,\dots,C_m$ and pairwise disjoint convex sets 
	$C'_1,\dots,C'_m$ such that $C'_i\subset C_i$ for each $i$, and
	\begin{enumerate}
		\item $\bigcup_{i=1}^m C'_i\subset K(t)\subset\bigcup_{i=1}^m C_i$,
		\item $V_d(C'_i)\gg t$ and $V_d(C_i)\ll t$ for each $i$,
		\item for each cap $C$ with $C\cap K(v>t)=\emptyset$, there is a $C_i$ containing $C$.
	\end{enumerate}
\end{proposition}

We conclude this section with a statement about the measure of the set of linear subspaces of
$\R^d$ that form a small angle with a fixed vector, which will be useful later.
\begin{lemma}\textup{\cite[Lemma 1]{Barany2010}}\label{lemma:measure}
	For fixed $z\in\SS^{d-1}$ and small $a>0$, 
	\begin{equation*}
	\nu_{\ell}(\{L\in G(d,\ell): \sphericalangle(z,L)\leq a\})=\Theta\bigl(a^{d-\ell}\bigr), \quad 
	\ell\in\{1,\dots,d\}.
	\end{equation*}
\end{lemma}

\subsection{Bound for normal approximation}
Let \((\Omega, \mathcal{A},\P)\) be a fixed probability space. We indicate with \(\E\) the expectation and with \(\Var\) the variance with respect to \(\P\).  Let \(X\) and \(Y\) be two real-valued random variables defined on \((\Omega, \mathcal{A},\P)\) with distributions \(\P_X\) and \(\P_Y\), respectively. The Kolmogorov distance between $\P_X$ and $\P_Y$ is defined by 
\begin{equation*}
d_K(\P_X,\P_Y)=\sup_{u\in\R}|\P(X\leq u)-\P(Y\leq u)|.  
\end{equation*}
With a slight abuse of notation, we write \(d_K(X,Y)\) to indicate \(d_K(\P_X,\P_Y)\). It is important to recall that the Kolmogorov distance is a metrization of the convergence in distribution, i.e., given a sequence of random variables \((X_n)_{n\in\mathbb{N}}\) and another random variable \(Y\) such that \(\lim\limits_{n\to\infty} d_K(X_n,Y)=0\), then \((X_n)_{n\in\mathbb{N}}\) converges in distribution to \(Y\).

\bigskip

Let \(S\) be a Polish space. Consider a function \(f\colon\cup_{k=1}^n S^k\to\R\) that acts on the point configurations of at most \(n\in\mathbb{N}\) points of \(S\). Let \(f\) be measurable and symmetric, i.e., invariant under permutations of the arguments. In the setting of this paper, \(S\) is the boundary of a smooth convex body, while \(f\) is an intrinsic volume of the convex hull of its arguments. Given a point \(x=(x_1,\ldots,x_h)\in\cup_{k=1}^n S^k\), we indicate with \(x^i\) the vector obtained from \(x\) by removing its \(i\text{-th}\) coordinate, namely \(x^i\coloneqq(x_1,\ldots,x_{i-1},x_{i+1},\ldots,x_h)\). Analogously, we define \(x^{ij}\coloneqq(x_1,\ldots,x_{i-1},x_{i+1},\ldots,x_{j-1},x_{j+1},\ldots,x_h)\). 

\bigskip

We now define the first- and second-order difference operators, applied to \(f\), as 
\begin{equation*}
D_i f(x)\coloneqq f(x)-f(x^i)\quad\text{and}\quad D_{i,j}f(x)\coloneqq f(x)-f(x^i)-f(x^j)+f(x^{ij}),
\end{equation*}
respectively. We indicate with \(X=(X_1,\ldots,X_n)\) a random vector of elements of \(S\) on \((\Omega, \mathcal{A},\P)\). Let \(X'\) and \(\tilde{X}\) be independent copies of \(X\). A vector \(Z=(Z_1,\ldots,Z_n)\) is called a recombination of \(\{X,X',\tilde{X}\}\), whenever \(Z_i\in\{X_i,X_i',\tilde{X}_i\}\) for every \(i\in\{1,\ldots,n\}\). For a subset \(A\subseteq\{1,\ldots,n\}\) of the index set, we write \(X^{A}=(X_1^A,\ldots, X_n^A)\) with
\begin{equation*}
	X_i^A\coloneqq\begin{cases}
		X_i&: i\notin A,\\
		X_i'&: i\in A.
	\end{cases}
\end{equation*}
In order to rephrase the normal approximation bound from \cite{LachiezeRey}, it is convenient to define the following quantities, namely,
\begin{align*}
\gamma_1 & :=\sup_{(Y,Y',Z,Z')}\mathbf{E}\bigl[\indicator\{D_{1,2}f(Y)\neq 0\}\,\indicator\{D_{1,3}f(Y')\neq 0\}\,D_2f(Z)^2\,D_3f(Z')^2\bigr]\,,\\
\gamma_2 & := \sup_{(Y,Z,Z')}\mathbf{E}\big[\indicator\{D_{1,2}f(Y)\neq 0\}\,D_1f(Z)^2\,D_2f(Z')^2\big]\,,\\
\gamma_3 & := \mathbf{E}\bigl[|D_{1}f(X)\big|^4\bigr]\,,\\
\gamma_4 &:= \mathbf{E}\bigl[|D_1f(X)|^3\bigr]\,,\\
\gamma_5 &:= \sup_{A\subseteq\{1,\ldots,n\}}\mathbf{E}\bigl[|f(X)D_1f(X^A)^3|\bigr]\,,
\end{align*}
where the suprema in the definitions of $\gamma_1$ and $\gamma_2$ run over all combinations of vectors 
$(Y,Y',Z,Z')$ or $(Y,Z,Z')$ that are recombinations of $\{X,X',\tilde{X}\}$. 

\bigskip

\begin{proposition}
	\label{prop:normalbound}
	\textup{\cite[Theorem 5.1]{LachiezeRey}} Let $W:=f(X_1,\dots,X_n)$ and assume that $\E[W]=0$ and $\E[W^2]<\infty$. Moreover, let \(N\) be a standard Gaussian random variable. Then, the following bound for the normal approximation holds:
\begin{equation*}
d_K\biggl(\frac{W}{\sqrt{\Var [W]}},N\biggr)\ll \frac{\sqrt{n}}{\Var [W]}(\sqrt{n^2\gamma_1}+\sqrt{n\gamma_2}
+\sqrt{\gamma_3})+\frac{n}{(\Var [W])^{\frac{3}{2}}}\gamma_4+\frac{n}{(\Var [W])^2}\gamma_5.
\end{equation*} 
\end{proposition}

\section{Lower variance bounds}
\label{Sec:lowerbound}
\subsection{Auxiliary geometric construction}\label{Sec:local-perturbations}

The following geometrical construction is taken from \cite[Section~3.1]{Richardson2008}. Let $E$ be the standard paraboloid given by
\begin{equation*}
E=\{z\in\R^d \ : \ z_d\geq z_1^2+\cdots+z_{d-1}^2\}.
\end{equation*}
We construct a simplex \(S\) in \(C^E(0,1)\) in the following way. The base is a regular simplex whose vertices \(v_1,\ldots,v_d\) lie on \(\partial E\cap H(e_d, 1/(3(d-1)^{2}))\) while \(v_0=(0,\ldots,0)\) is the apex of \(S\). Notice that the ball \(2E\cap H(e_d,1)\) has radius \(\sqrt{2}\), while the inradius of the base of the simplex is \(1/(\sqrt{3}(d-1)^2) \) and therefore, \(\{\lambda z\in\R^d : \lambda\ge 0, z\in S \}\cap H(e_d,1)\) has inradius \(3(d-1)^2/(\sqrt{3}(d-1)^2)=\sqrt{3}\). In particular, this implies that 
\begin{equation*}
\{\lambda z\in\R^d : \lambda\ge 0, z\in S \}\supseteq 2E\cap H(e_d,1).
\end{equation*}
For \(i\in\{0,1,\ldots,d\}\), let \(v_i'\) be the orthogonal projection of \(v_i\) onto \(\mathrm{span}\{e_1,\ldots,e_{d-1}\}\). Consider \(B_0\coloneqq B^{d-1}(v_0',r)\subseteq\R^{d-1}\) and \(B_i\coloneqq B^{d-1}(v_i',r')\subseteq\R^{d-1}\), $i\in\{1,\dots,d\}$, for some radii \(r\) and \(r'\) to be chosen later. Let \(b_E\) be the quadratic form associated with \(E\), i.e., \(b_E(y)=\lVert y\rVert^2\) for \(y\in\R^{d-1}\). For \(i\in\{0,\ldots,d\}\), we can define the lift \(B_i'\coloneqq\tilde b(B_i)\) on \(\partial E\) of the sets \(B_i\), where \(\tilde b\) indicates the mapping
\begin{equation*}
\tilde b\colon\R^{d-1}\to\partial E,\quad
y\mapsto (y,b_E(y)).
\end{equation*}
Note that, if \(r\) and $r'$ are small enough, then, by continuity, for any \((d+1)\text{-tuple}\) of points \(x_i\in B_i'\), the following still holds,
\begin{equation}
\label{eq:circleincone}
\{\lambda z\in\R^d : \lambda\ge 0, z\in [x_0,\ldots,x_d] \}\supseteq 2E\cap H(e_d,1).
\end{equation}
Then, we extend the aforementioned argument to arbitrary caps of \(\partial K \). For each point \(x\in\partial K\), we consider the approximating paraboloid \(Q_x\) of \(K\) at \(x\). Let \(T_x(K)\) be the tangent space of \(K\) at the point \(x\). The space \(T_x(K)\) can be identified with \(\R^{d-1}\) having \(x\) as its origin. Then, there exists a unique affine map \(A_x\) such that \(A_x(C^E(0,1))=C^{Q_x}(x,h)\) while mapping the coordinate axes onto the coordinate axes of \(T_x(K)\times\R \). We define \(D_i(x)\coloneqq A_x(B_i) \), $i\in\{0,\dots,d\}$. Then, it is true that
\(\vol_{d-1}(D_i(x))=c_1\,h^{\frac{d-1}{2}} \) for a constant $c_1>0$. We define now \(D_i'(x)\coloneqq \tilde f^{(x)}(D_i(x))\), for a neighbourhood \(U\subseteq T_x(K)\) of \(x\),
\begin{equation*}
\tilde f^{(x)} \colon U\to\partial K,\quad
y\mapsto (y,f^{(x)}(y)).
\end{equation*}
Since \(K\in\mathcal{K}_{+}^2\), there exist positive lower and upper bounds for the curvature. Thus, due to the curvature bounds of \(K\), it holds that
\begin{equation}\label{eq:hausdorff-measure}
c_K h^{\frac{d-1}{2}}\le \mathcal{H}^{d-1}(D_i'(x))\le C_K h^{\frac{d-1}{2}},
\end{equation}
where \(c_K\) and \(C_K\) are positive constants depending only on \(K\). 

\bigskip

By continuity, if every \(x_i\) belongs to a ball \(B^d(v_i,\eta)\), \eqref{eq:circleincone} is preserved whenever \(\eta>0\) is small enough.
Moreover, we can choose \(r\) and \(r'\) to be small enough such that for every \(x\in\partial K\) and every \(i\in\{0,\ldots, d\}\), \(D'_i(x)\subseteq A_x(B^d(v_i,\eta))\).
Indeed, define for \(\varepsilon>0\) and every \(i\in\{0,\ldots,d\}\), the set \(U_i=\{(x,y) \in \R^d : x \in B^{d-1}(\proj_{\R^{d-1}} v_i,\eta/2), y \in [ (1+\varepsilon)^{-1} b_E(x) , (1+\varepsilon) b_E(x) ] \} \). If \(\varepsilon\) is small enough, then \(U_i \subseteq B^d(v_i,\eta) \). Using \Cref{lemma:approx}, we can take \(h\) small enough such that \( (1+\varepsilon)^{-1} b_{Q_x}(y) \le f^{(x)}(y) \le (1+\varepsilon) b_{Q_x}(y)  \). In particular, if we choose \(r,r' < \eta/2\), then \(D_i'(x)\subseteq A_x(U_i) \subseteq A_x(B^d(v_i,\eta))  \). As a consequence, we have, for \(x_i\in D_i'(x)\),
\begin{equation}
\label{eq:Ksectionincone}
\{\lambda z\in\R^d : \lambda\ge 0, z\in [x_0,\ldots,x_d] \}\supseteq 2Q_x\cap H(u_x,h_K(u_x)-h)\supseteq K\cap H(u_x,h_K(u_x)-h),
\end{equation}
where the last inclusion holds whenever \(h\le h_0\) for \(h_0\) sufficiently small. Therefore, from now on \(r,r'\) and \(h_0\) are chosen such that the previous argument holds true.

\subsection{Proof of the lower bounds}
In this section we combine tools from  \cite{Barany2010, Reitzner2005, Richardson2008}.
Let $K\in\mathcal{K}_+^2$ and $X_1,\dots,X_n$ be independent random points that are chosen from $\partial K$ 
according to the probability distribution $\mathcal{H}^{d-1}$. Due  to
%Lemma \ref{lemma:cap-covering}
\cite[Lemma~13]{Reitzner2005}, we can choose $n$ points $y_1,\dots,y_n\in\partial K$ and corresponding disjoint caps of $K$, namely,
$C^K(y_j,h_n)$ for $j\in\{1,\dots,n\}$, with  $h_n=\Theta\bigl(n^{-\frac{2}{d-1}}\bigr)$. For all $i\in\{0,\dots,d\}$ and $j\in\{1,\dots,n\}$, we define the sets $\{D_i(y_j)\}$ and $\{D_i'(y_j)\}$ as in \Cref{Sec:local-perturbations}. Let $A_j$, $j\in\{1,\dots,n\}$, be the event that exactly one random point is contained in each $D'_i(y_j)$, $i\in\{0,\dots,d\}$, and every other point is outside of 
$C^K(y_j,h_n)\cap\partial K$. 

\begin{lemma}\textup{\cite[Section~3.2]{Richardson2008}}\label{lemma:probability-of-event}
	For \(n\) large enough, and all $j\in\{1,\dots,n\}$, there exists a constant $c\in(0,1)$ such that 
	$\P(A_j)\geq c$.
\end{lemma}
\begin{proof}
	The probability of the event $A_j$ is
	\begin{align*}
	\P(A_j)&=\binom{n}{d+1}\P(X_i\in D'_i(y_j), i\in\{0,\dots,d\})
	\P(X_i\notin C^K(y_j,h_n)\cap\partial K, i\in\{d+1,\dots,n\}) \\
	&=\binom{n}{d+1}\prod_{i=0}^d \mathcal{H}^{d-1}(D'_i(y_j)) \prod_{k=d+1}^n
	(1-\mathcal{H}^{d-1}(C^K(y_j,h_n)\cap\partial K)).
	\end{align*}
	Combining \Cref{relation-of-caps}, \cite[Lemma~13]{Reitzner2005}
	%\Cref{lemma:cap-covering} 
	and \Cref{eq:hausdorff-measure}, we obtain
	\begin{equation*}
	\P(A_j)\geq c_2n^{d+1}n^{-d-1}(1-c_3n^{-1})^{n-d-1}\geq c>0,
	\end{equation*}
	where all constants are positive. 
\end{proof}

Let $\mathcal{F}$ be the $\sigma$-field generated by the positions of all $X_1,\dots,X_n$ 
except those which are contained in $D'_0(y_j)$ with $\indicator_{A_j}=1$. 
Assume that $\indicator_{A_j}=\indicator_{A_k}=1$ for some 
$j,k\in\{1,\dots,n\}$ and without loss of generality that $X_j$ and $X_k$ are 
the points in $D'_0(y_j)$ and $D'_0(y_k)$. By \Cref{eq:Ksectionincone}, it is not possible 
that there is an edge between $X_j$ and $X_k$. Therefore, the change of the 
intrinsic volume affected by moving $X_j$ within $D'_0(y_j)$ is independent 
of the change of the intrinsic volume of moving $X_k$ within  $D'_0(y_k)$. 
As a consequence, we obtain
\begin{equation*}
\Var[V_{\ell}(K_n)|\mathcal{F}]=\sum_{j=1}^n\Var_{j}[V_{\ell}(K_n)]
\indicator_{A_j},
\end{equation*}
where the variances \(\Var_{j}[\cdot]\) are taken over $X_j\in D'_0(y_j)$. 

\bigskip

For $j\in\{1,\dots,n\}$ and $i\in\{0,\dots,d\}$, let $z_j^{i}$ be an arbitrary 
point in $D'_i(y_j)$. We indicate with $N_j$ the normal cone of the simplex 
$[z_j^0,\dots,z_j^d]$ at vertex $z_j^0$. Let $S_j$ be the cone with base $H(u_{z^0_j},h_K(u_{z_j^0})-h_n)\cap 2Q_x$ 
and vertex $z_j^0$. Note that $u_{z_j^0}$ is the unique unit outer normal of $K$ at $z_j^0$. The corresponding 
normal cone of $S_j$ at $z_j^0$ is denoted by $\bar{N}_j$. Moreover, the angular aperture of \(S_j\) at \(z_j^0\) is at most $ c'_K\sqrt{h_n}$, where \(c'_K>0\) is a constant that depends on $K$. Because of this and \Cref{eq:Ksectionincone},  we can find sets $\Sigma_j$ such that
\begin{equation}\label{eq:normal-cone}
  \SS^{d-1}\cap N_j\subset\SS^{d-1}\cap \bar{N}_j\subset\SS^{d-1}\cap(u_{z_j^0}+
  c'_K\sqrt{h_n}B^d)=:\Sigma_j.
\end{equation}
We fix $j\in\{1,\dots,n\}$ and $z_j^i\in D'_i(y_j)$ for all $i\in\{1,\dots,d\}$. Let 
$F_j:=[z_j^1,\dots, z_j^d]$ and define
\begin{equation*}
\widetilde{V}_{\ell}(z;F_j):=\binom{d}{\ell}\frac{\kappa_d}{\kappa_{\ell}\kappa_{d-\ell}}
\int_{G(d,\ell)}\indicator_{\{L\cap\Sigma_j\neq\emptyset\}}
\vol_{\ell}\bigl([z,F_j] |L\bigr)\nu_{\ell}(\dint L), \quad z\in D'_0(y_j), \ell\in\{1,\dots,d\}.
\end{equation*}

\begin{lemma}\label{lemma:lower-bound}
	Let $j\in\{1,\dots,n\}$ and let $X_j$ be a point chosen 
	with respect to the normalized Hausdorff measure restricted to $D'_0(y_j)$. Then,
	\begin{equation*}
	\Var_j[\widetilde{V}_{\ell}(X_j;F_j)]=\Theta\bigl(n^{-2\frac{d+1}{d-1}}\bigr), \quad \ell\in\{1,\dots,d\}.
	\end{equation*}
\end{lemma}
\begin{proof}
	%  Let $E$ be the standard paraboloid in $\R^d$, namely
	%  \begin{equation*}
	%    E=\{x\in\R^d | \langle x,e_d\rangle\geq \sum_{j=1}^{d-1}\langle x,e_j\rangle^2\}.
	%  \end{equation*}
	%  It is obvious that the cap $C(0,h)$ of $E$ has a $(d-1)$-dimensional base with radius 
	%  $h^{1/2}$ and the height from the base to the origin is $h$. Since the approximating 
	%  paraboloid $Q(x)$ is the image of an affine map of $E$ which leaves the coordinate 
	%  axis invariant, the behaviour of the base of the cap of $Q(x)$ is up to a constant of the same 
	%  order as for the standard paraboloid. 
	Note that 
	$[X_j,F_j] |L$ is a simplex in $L\in G(d,\ell)$ with base $F_j|L$ and additional point $X_j|L$.
	As a consequence, the height of $[X_j,F_j] |L$ is proportional to $h_n$ and
	\begin{equation*}
	\vol_{\ell-1}(F_j | L)=\Theta\bigl(h_n^{\frac{\ell-1}{2}}\bigr),
	\end{equation*}
	where $L\in G(d,\ell)$ with $L\cap\Sigma_j\neq\emptyset$. Thus,
	\begin{equation*}
	\vol_{\ell}\bigl([X_j,F_j] |L\bigr)=\Theta\bigl(h_n^{\frac{\ell+1}{2}}\bigr).
	\end{equation*}
	Due to Lemma \ref{lemma:measure} and \Cref{eq:normal-cone}, it follows
	\begin{equation*}
	\int_{G(d,\ell)}\indicator_{\{L\cap\Sigma_j\neq\emptyset\}}
	\,\nu_{\ell}(\dint L)=\nu_{\ell}(\{L\in G(d,\ell):L\cap\Sigma_j\neq\emptyset\})
	=\Theta\bigl(h_n^{\frac{d-\ell}{2}}\bigr).
	\end{equation*}
	Therefore, we obtain
	\begin{equation*}
	\widetilde{V}_{\ell}(X_j;F_j)=\Theta\bigl(h_n^{\frac{d+1}{2}}\bigr).
	\end{equation*}
	Let $X_j^1$ and $X_j^2$ be independent copies of $X_j$, then
	\begin{equation*}
	|\widetilde{V}_{\ell}(X_j^1;F_j)-\widetilde{V}_{\ell}(X_j^2;F_j)|=\Theta\bigl(h_n^{\frac{d+1}{2}}\bigr),
	\end{equation*}
	since the heights of $X_j^1|L$ and $X_j^2|L$ are 
	different with probability 1.
	Using $h_n=\Theta\bigl(n^{-\frac{2}{d-1}}\bigr)$, we obtain
	\begin{equation*}
	\begin{split}
	\Var_j\bigl[\widetilde{V}_{\ell}(X_j;F_j)\bigr]
	&=\frac{1}{2}\E\Bigl[\bigl|\widetilde{V}_{\ell}(X_j^1;F_j)-\widetilde{V}_{\ell}(X_j^2;F_j)\bigr|^2\Bigr] \\
	&=\Theta\bigl(n^{-2\frac{d+1}{d-1}}\bigr). \qedhere
	\end{split}
	\end{equation*}
\end{proof}

We can now proceed with the proof of the lower variance bounds.
\begin{proof}[Proof of the lower bounds of \Cref{thm:variancebounds}]
	Let $\mathcal{F}$ be the $\sigma$-field defined as above. The conditional variance 
	formula implies that
	\begin{equation*}
	\Var[V_{\ell}(K_n)]=\E[\Var[V_{\ell}(K_n)|\mathcal{F}]]+\Var
	[\E[V_{\ell}(K_n)|\mathcal{F}]]\geq \E[\Var[V_{\ell}(K_n)|\mathcal{F}]].
	\end{equation*}
	As already mentioned, $\mathcal{F}$ induces an independence property. Therefore, we 
	obtain
	\begin{equation*}
	\Var[V_{\ell}(K_n)|\mathcal{F}]=\sum_{j=1}^n\Var_{j}[V_{\ell}(K_n)]\indicator_{A_j}
	=\sum_{j=1}^n\Var_j[\widetilde{V}_{\ell}(X_j;F_j)]\indicator_{A_j}.
	\end{equation*}
	Finally, applying \Cref{lemma:probability-of-event}, \Cref{lemma:lower-bound} and 
	taking expectation yields
	\begin{equation*}
	\Var[V_{\ell}(K_n)]\gg n^{-2\frac{d+1}{d-1}}\sum_{j=1}^n\P(A_j)
	\gg n^{-2\frac{d+1}{d-1}} n=n^{-\frac{d+3}{d-1}}.\qedhere
	\end{equation*}
\end{proof}

\section{Upper variance bounds}
\label{Sec:upperbounds}
In the following, we find upper bounds for $\Var[V_{\ell}(K_n)]$, $\ell\in\{1,\dots,d\}$. 
%The result is presented for the case $K=B^d$. But it can be transferred to general smooth convex bodies. 
The proof is based on the Efron-Stein jackknife inequality and follows the ideas of \cite{Barany2010}. In
contrast to \cite{Barany2010}, we use the concept of surface body, in particular, \Cref{lemma:surface-body} 
about the fact that the surface body is contained in the random polytope $K_n$ with high 
probability. Moreover, we make use of \Cref{relation-of-caps} for our estimates.

\begin{proof}[Proof of the upper bounds of \Cref{thm:variancebounds}]
	First, let $K=B^d$. We indicate with $T_n$ the event that the surface body $K(s\geq \tau_n)$ is 
	contained in $K_n$. Let \( \ell\in\{1,\ldots,d\}\). Applying the Efron-Stein jackknife inequality yields
	\begin{equation}\label{eq:splitting}
	\begin{split}
	\Var[V_{\ell}(K_n)]&\ll n\E\bigl[(V_{\ell}(K_{n+1})-V_{\ell}(K_n))^2\bigr] \\ 
	&=n\E\bigl[(V_{\ell}(K_{n+1})-V_{\ell}(K_n))^2\indicator_{T_n}\bigr]+
	n\E\bigl[(V_{\ell}(K_{n+1})-V_{\ell}(K_n))^2\indicator_{T_n^c}\bigr].
	\end{split}
	\end{equation}
	 It is obvious that $(V_{\ell}(K_{n+1})-V_{\ell}(K_n))^2\leq V_{\ell}(K)^2$ and 
	$\E[\indicator_{T_n^c}]=\P(T_n^c)$. Since the parameter $\alpha$ can be chosen arbitrarily 
	big in \Cref{lemma:surface-body}, the second term in \Cref{eq:splitting} is negligible in the asymptotic analysis.
	By \Cref{eq.Kubota}, we obtain
	\begin{equation}\label{estimate-Kubota}
	\begin{split}
	\Var[V_{\ell}(K_n)]
	&\ll n\E\bigl[(V_{\ell}(K_{n+1})-V_{\ell}(K_n))^2\indicator_{T_n}\bigr]\\
	&\ll n\E\biggl[\int_{G(d,\ell)}\vol_{\ell}((K_{n+1}|A)\setminus (K_n |A))\nu_{\ell}(\dint A) \\
	&\qquad\qquad\times\int_{G(d,\ell)}\vol_{\ell}((K_{n+1}|B)\setminus (K_n |B))\,\nu_{\ell}(\dint B) \indicator_{T_n}\biggr] \\
	&\ll n\E\biggl[\int_{G(d,\ell)}\int_{G(d,\ell)}\vol_{\ell}((K_{n+1}|A)\setminus (K_n |A))\,\vol_{\ell}((K_{n+1}|B)\setminus (K_n |B))\\
	&\qquad\qquad\times\indicator_{T_n}\nu_{\ell}(\dint A)\nu_{\ell}(\dint B)\biggr].
	\end{split}
	\end{equation}
	If $X_{n+1}|A\in K_n|A$, then the set $(K_{n+1}|A)\setminus (K_n|A)$ is clearly empty. Otherwise, 
	$(K_{n+1}|A)\setminus (K_n|A)$ consists of several disjoint simplices which are the convex hull of 
	$X_{n+1}|A$ and those facets of $K_n|A$ that can be ``seen'' from $X_{n+1}|A$. For
	$I=\{i_1,\dots,i_{\ell}\}\subset\{1,\dots,n\}$, we indicate with $F_I$ the convex hull of 
	$X_{i_1},\dots,X_{i_{\ell}}$. Note that $F_I$ and $F_I|A$ are $(\ell-1)$-dimensional simplices 
	with probability $1$. The closed half space in $\R^d$ which is determined by the hyperplane 
	$A^{\bot}+\aff F_I$ and contains the origin is denoted by $H_0(F_I,A)$. 
	The other half space is $H_+(F_I,A)$. The corresponding $\ell$-dimensional half spaces in $A$ 
	are denoted by $H_0(F_I|A)$ and $H_+(F_I|A)$. 
	Let $\mathcal{F}(A)$ be the set of $(\ell-1)$-dimensional facets of $K_n|A$ that can be seen 
	from $X_{n+1}|A$. It is defined by
	\begin{equation*}
	\mathcal{F}(A)=\{F_I|A : K_n|A\subset H_0(F_I|A),\, X_{n+1}|A\in H_{+}(F_I|A), 
	I=\{i_1,\dots,i_{\ell}\}\subset \{1,\dots,n\}\}.
	\end{equation*}
	Then,
	\begin{equation}
	\begin{split}
	\label{estimate-Kubota2}
	\eqref{estimate-Kubota}
%	&\leq n\,\E\biggl[\int_{G(d,\ell)}\int_{G(d,\ell)}\vol_{\ell}((K_{n+1}|A)\setminus (K_n |A)) \, 
%	\vol_{\ell}((K_{n+1}|B)\setminus (K_n |B)) \\
%	&\qquad\qquad\times\nu_{\ell}(\dint A)\nu_{\ell}(\dint B) \indicator_{T_n}\biggr] \\
	&\ll n\int_{\SS^{d-1}}\dots\int_{\SS^{d-1}}\int_{G(d,\ell)}\int_{G(d,\ell)}
	\biggl(\sum_{F\in\mathcal{F}(A)}\vol_{\ell}([x_{n+1}|A,F])\biggr) \\
	&\qquad\times
	\biggl(\sum_{F^{'}\in\mathcal{F}(B)}\vol_{\ell}([x_{n+1}|B,F^{'}])\indicator_{T_n}\biggr)
	\nu_{\ell}(\dint A)\nu_{\ell}(\dint B)\mathcal{H}^{d-1}(\dint x_1)\cdots\mathcal{H}^{d-1}(\dint x_{n+1}).
	\end{split}
	\end{equation}
	Next, the integration is extended over all possible index sets $I,J$ and the order of integration is changed. 
	As a consequence, we obtain
	\begin{equation*}
	\begin{split}
	\eqref{estimate-Kubota2}&\ll n\int_{G(d,\ell)}\int_{G(d,\ell)}\int_{(\SS^{d-1})^{n+1}}
	\left(\sum_{I}\indicator\{F_I|A\in\mathcal{F}(A)\}\,\vol_{\ell}([F_I,x_{n+1}]|A)\right) \\
	&\qquad\qquad\times \left(\sum_{J}\indicator\{F_J|B\in\mathcal{F}(B)\}\,\vol_{\ell}([F_J,x_{n+1}]|B)
	\indicator_{T_n}\right) \\
	&\qquad\qquad\times \mathcal{H}^{d-1}(\dint x_1)\cdots\mathcal{H}^{d-1}(\dint x_{n+1})\nu_{\ell}(\dint A)\nu_{\ell}(\dint B).
	\end{split}
	\end{equation*}
	Note that $[F_I,X_{n+1}]|A$ and $[F_J,X_{n+1}]|B$ are contained in the associated caps 
	$C_{\ell}(I,A):=H_+(F_I,A)\cap B^d$ and $C_{\ell}(J,B)$. Moreover, we use the abbreviation
	\(	C_d(I,A)=(H_+(F_I|A)+A^{\bot})\cap B^d\).
	We indicate with $V_{\ell}(I,A)=\vol_{\ell}(C_{\ell}(I,A))$ and 
	$V_d(I,A)=V_d(C_d(I,A))$ the volumes of these caps.
	Therefore, the variance is bounded by
	\begin{align*}
	\Var[V_{\ell}(K_n)]&\ll n\sum_I \sum_J\int_{G(d,\ell)}\int_{G(d,\ell)} \int_{(\SS^{d-1})^{n+1}}
	\!\!\!\!\indicator\{F_I|A\in\mathcal{F}(A)\}V_{\ell}(I,A)\indicator\{F_J|B\in\mathcal{F}(B)\} \\
	&\qquad\times V_{\ell}(J,B)\,\indicator_{T_n}\,
	\mathcal{H}^{d-1}(\dint x_1)\cdots\mathcal{H}^{d-1}(\dint x_{n+1})\,\nu_{\ell}(\dint A)\,\nu_{\ell}(\dint B),
	\end{align*}
	where the summation extends over all $\ell$-tuples $I$ and $J$. Of course, these tuples may have
	a non-empty intersection. However, if the size of $I\cap J$ is fixed to be $k$, then the 
	corresponding terms in the sum are independent of the choice of $i_1,\dots,i_{\ell}$ and 
	$j_1,\dots,j_{\ell}$. For any $k\in\{0,1,\dots,\ell\}$, we indicate with $F$ the convex hull of 
	$X_1,\dots,X_{\ell}$ and by $G$ the convex hull of $X_{\ell-k+1},\dots,X_{2\ell-k}$. As
	in \cite{Barany2010}, we obtain
	\begin{equation}
	\label{eq:sigmak}
	\begin{split}
	\Var[V_{\ell}(K_n)]&\ll n\sum_{k=0}^{\ell} \binom{n}{\ell}\binom{\ell}{k}\binom{n-\ell}{\ell-k}
	\int_{G(d,\ell)}\int_{G(d,\ell)} \int_{(\SS^{d-1})^{n+1}}
		\!\!\!\!\indicator\{F_I|A\in\mathcal{F}(A)\}V_{\ell}(I,A) \\
	&\qquad\times \indicator\{F_J|B\in\mathcal{F}(B)\}V_{\ell}(J,B)\,\indicator_{T_n}\,
	\mathcal{H}^{d-1}(\dint x_1)\cdots\mathcal{H}^{d-1}(\dint x_{n+1})\,\nu_{\ell}(\dint A)\,\nu_{\ell}(\dint B).
	\end{split}
	\end{equation}
	We indicate with $\Sigma_k$ the $k\text{-th}$ term in the previous sum. By symmetry, we can restrict the summation to those tuples where $V_d(I,A)\geq V_d(J,B)$. In 
	addition to that, we multiply the integrand by $\indicator\{C_d(I,A)\cap C_d(J,B)\neq\emptyset\}$. 
	This is indeed possible because the caps have at least the point $X_{n+1}$ in common. 
	It follows immediately that
	\begin{align*}
	\Sigma_k &\ll n^{2\ell-k+1}\int_{G(d,\ell)}\int_{G(d,\ell)}
	\int_{(\SS^{d-1})^{n+1}}\!\!\!\!\indicator\{F|A\in\mathcal{F}(A)\}\,V_{\ell}(I,A)\, \indicator\{C_d(I,A)\cap C_d(J,B)\neq\emptyset\} \\
	&\qquad\qquad\times\indicator\{G|B\in\mathcal{F}(B)\} \,V_{\ell}(J,B)\,
	\indicator\{V_d(I,A)\geq V_d(J,B)\}\\
	&\qquad\qquad\times \indicator_{T_n}\,
	\mathcal{H}^{d-1}(\dint x_1)\cdots\mathcal{H}^{d-1}(\dint x_{n+1})\,\nu_{\ell}(\dint A)\,\nu_{\ell}(\dint B).
	\end{align*}
	Next, we integrate with respect to $x_{2\ell-k+1},\dots,x_{n},x_{n+1}$. Due to the 
	condition $F|A\in\mathcal{F}(A)$, the points $X_{2\ell-k+1},\dots,X_n$ are contained in 
	$H_0(F,A)$ and $X_{n+1}$ is in $H_+(F,A)$. Therefore,
	\begin{align*}
	\Sigma_k &\ll n^{2\ell-k+1}\int_{G(d,\ell)}\int_{G(d,\ell)}\int_{(\SS^{d-1})^{2\ell-k}}
	\!\!(1-\mathcal{H}^{d-1}(C_d(I,A)\cap\SS^{d-1}))^{n-2\ell+k} \\
	&\qquad\times \mathcal{H}^{d-1}(C_d(I,A)\cap\SS^{d-1})\,V_{\ell}(I,A)\,
	\indicator\{C_d(I,A)\cap C_d(J,B)\neq\emptyset\,\}V_{\ell}(J,B) \\
	&\qquad\times \indicator\{V_d(I,A)\geq V_d(J,B)\}\,\indicator_{T_n}
	\,\mathcal{H}^{d-1}(\dint x_1)\cdots\mathcal{H}^{d-1}(\dint x_{2\ell-k})\,\nu_{\ell}(\dint A)\,\nu_{\ell}(\dint B).
	\end{align*}
	The assumption $V_d(I,A)\geq V_d(J,B)$ implies that the height of the cap $C_d(I,A)$ 
	is at least the height of $C_d(J,B)$.  Due to $C_d(I,A)\cap C_d(J,B)\neq\emptyset$, 
	we find a constant $\beta$ such that $C_d(J,B)$ is contained in $\beta\,C_d(I,A)$. 
	More precisely, $\beta\,C_d(I,A)$ is an enlarged homothetic copy of $C_d(I,A)$, where the 
	center of homothety $z\in\SS^{d-1}$ coincides with the center of the cap $C_d(I,A)$.
	It follows from the homogeneity that the Hausdorff measure (restricted to $\beta\,\SS^{d-1}$) 
	of $\beta\,C_d(I,A)$ is up to a constant $\mathcal{H}^{d-1}(C_d(I,A)\cap\SS^{d-1})$.
	Therefore,
	\begin{align*}
	\int_{(\SS^{d-1})^{\ell-k}}\!\!
	&\indicator\{C_d(I,A)\cap C_d(J,B)\neq\emptyset\}\,\indicator\{V_d(I,A)\geq V_d(J,B)\}\\
	&\times V_{\ell}(J,B)\,\mathcal{H}^{d-1}(\dint x_{\ell+1})\cdots\mathcal{H}^{d-1}(\dint x_{2\ell-k})
	\ll\mathcal{H}^{d-1}(C_d(I,A)\cap\SS^{d-1})^{\ell-k}V_{\ell}(I,A).
	\end{align*}
	As in \cite{Barany2010}, the conditions $C_d(I,A)\cap C_d(J,B)\neq\emptyset$ 
	and $V_d(I,A)\geq V_d(J,B)$ are only satisfied if the angle between $z$ and the subspace $B$ 
	is not larger than twice the central angle \(\delta\) of the cap $C_d(I,A)$. Moreover, \(\delta\) is bounded by 
	\begin{equation}
	\label{eq:boundbeta}
	\delta\ll V_d(I,A)^{1/(d+1)}.
	\end{equation}
	Thus,
	\begin{align*}
	\Sigma_k &\ll n^{2\ell-k+1}\int_{G(d,\ell)}\int_{G(d,\ell)}\int_{(\SS^{d-1})^{\ell}}
	\!\!(1-\mathcal{H}^{d-1}(C_d(I,A)\cap\SS^{d-1}))^{n-2\ell+k}  \\
	&\qquad\times \mathcal{H}^{d-1}(C_d(I,A)\cap\SS^{d-1})^{\ell-k+1}  V_{\ell}(I,A)^2 \,\indicator\{\sphericalangle (z,B)\ll V_d(I,A)^{1/(d+1)}\}\\
	&\qquad\times \indicator_{T_n}\,\mathcal{H}^{d-1}(\dint x_1)\cdots\mathcal{H}^{d-1}(\dint x_{\ell})\,\nu_{\ell}(\dint A)\,\nu_{\ell}(\dint B).
	\end{align*}
	Due to Lemma \ref{relation-of-caps}, the condition $T_n$ can be replaced by the condition 
	\begin{displaymath}
	V_d(I,A)\leq c_1\, (\log{n}/n)^{(d+1)/(d-1)}
	\end{displaymath} 
	for some constant $c_1>0$. In the following, the economic cap covering theorem is used, 
	recall \Cref{economic-cap-theorem}. Let $h$ be a positive integer such that $2^{-h}\leq \log{n}/n$. 
	Note that the smallest possible value of $h$ is $h_0=-\lfloor \log_2 (\log{n}/n) \rfloor$.
	According to the economic cap covering theorem, we find for each $h$ a collection of 
	caps $\{C_1,\dots,C_{m(h)}\}$ which cover the wet part of $B^d|A$ with parameter 
	$(2^{-h})^{(\ell+1)/(d-1)}$. This collection of caps is denoted by $\mathcal{M}_h$. 
	Each cap $C_i$ can be viewed as a projection of a $d$-dimensional cap $C_i(A)$ from 
	$B^d$ to $A$. Now we consider an arbitrary tuple $(X_1,\dots,X_{\ell})$ which has a corresponding 
	cap $C_d(I,A)$ having volume at most $c_1\,(\log{n}/n)^{(d+1)/(d-1)}$. We relate to 
	$(X_1,\dots,X_{\ell})$ the maximal $h$ such that $C_{\ell}(I,A)\subset C_i$ for some 
	$C_i\in\mathcal{M}_h$. This is indeed possible since at least $2^{-h_0}$ is roughly
	$\log{n}/n$ and the volume of the caps in $\mathcal M_h$ tends to zero as 
	$h\to\infty$. As a consequence, 
	we obtain
	\begin{equation*}
	V_{\ell}(I,A)\leq \vol_{\ell}(C_i)\ll 2^{-h(\ell+1)/(d-1)}
	\end{equation*}
	and 
	\begin{equation*}
	V_d(I,A)\leq V_d(C_i(A))\ll2^{-h(d+1)/(d-1)}.
	\end{equation*}
	
	According to \Cref{relation-of-caps}, \(\mathcal{H}^{d-1}(C_d(I,A)\cap\SS^{d-1})\leq\mathcal{H}^{d-1}(C_i(A)\cap\SS^{d-1})\ll 2^{-h}\).
	Due to the maximality of $h$, it holds \(
	V_d(I,A)\geq 2^{-(h+1)(d+1)/(d-1)}\). In addition to that, it follows from \Cref{relation-of-caps} that
	\(\mathcal{H}^{d-1}(C_d(I,A)\cap\SS^{d-1})\geq c_2 2^{-(h+1)}\),
	for some constant $c_2>0$.
	Therefore, we obtain
	\begin{align*}
	&(1-\mathcal{H}^{d-1}(C_d(I,A)\cap\SS^{d-1}))^{n-2\ell+k}\mathcal{H}^{d-1} (C_d(I,A)\cap\SS^{d-1})^{\ell-k+1} V_{\ell}(I,A)^2 \\
	&\qquad\qquad\ll (1-c_2 2^{-(h+1)})^{n-2\ell+k}2^{-h(\ell-k+1)}2^{-2h(\ell+1)/(d-1)}.
	\end{align*}
	Then, we integrate each $(X_1,\dots,X_{\ell})$ on $(C_i(A))^{\ell}$ and 
	we use the fact $1-x\leq \exp(-x)$ to obtain 
	\begin{align*}
	&\exp(-c_2(n-2\ell+k)2^{-h-1})2^{-h(\ell-k+1)}2^{-2h(\ell+1)/(d-1)}\mathcal{H}^{d-1}(C_i(A)\cap\SS^{d-1})^{\ell} \\
	&\qquad\qquad\ll\exp(-c_2(n-2\ell+k)2^{-h-1})2^{-h(\ell-k+1)}2^{-2h(\ell+1)/(d-1)}2^{-h\ell}.
	\end{align*}
	
	Since the volume of the wet part of $B^{\ell}$ with parameter $2^{-h(\ell+1)/(d-1)}$ is 
	%$\vol_{\ell}(B^{\ell}(v\leq 2^{-h(\ell+1)/(d-1)}=
	$\Theta\bigl(2^{-2h/(d-1)}\bigr)$ (note that $h\to\infty$, as $ n\to\infty$), we obtain 
	\begin{equation}\label{number-caps}
	|\mathcal{M}_h|\ll \frac{2^{-2h/(d-1)}}{2^{-h(\ell+1)/(d-1)}}=2^{h(\ell-1)/(d-1)}.
	\end{equation}
	Finally, this results in
	\begin{align*}
	\int_{G(d,\ell)}&\int_{(\SS^{d-1})^{\ell}}
	(1-\mathcal{H}^{d-1}(C_d(I,A)\cap\SS^{d-1}))^{n-2\ell+k}\mathcal{H}^{d-1}(C_d(I,A)\cap\SS^{d-1})^{\ell-k+1} V_{\ell}(I,A)^2 \\
	&\qquad \times \indicator\{\sphericalangle(z,B)\ll V_d(I,A)^{1/(d+1)}\}
	\indicator_{T_n}\,\mathcal{H}^{d-1}(\dint x_1)\cdots\mathcal{H}^{d-1}(\dint x_{\ell})\nu_{\ell}(\dint B) \\
	&\ll\sum_{h=h_0}^{\infty}\exp(-c_2(n-2\ell+k)2^{-h-1})2^{-h(\ell-k+1)}2^{-2h(\ell+1)/(d-1)}2^{-h\ell} \\
	&\qquad\times |\mathcal{M}_h|\nu_{\ell}(\{\sphericalangle(z,B)\ll V_d(I,A)^{1/(d+1)}\})  \\
	&\ll \sum_{h=h_0}^{\infty}\exp(-c_2(n-2\ell+k)2^{-h-1})2^{-h[(2\ell-k+1)+(d+3)/(d-1)]}.
	\end{align*}
	
	Note that we used \Cref{lemma:measure} and \Cref{number-caps} in the last step. 
	As in \cite{Barany2010}, we divide the previous sum into two parts in order to see the 
	magnitude of the variance. The integer $h_1$ is defined by
	\begin{equation*}
	2^{-h_1}\leq \frac{1}{n}<2^{-h_1+1}.
	\end{equation*}
	On the one hand, we have
	\begin{align*}
	\sum_{h=h_1}^{\infty}\exp(-c_2(n-2\ell+k)2^{-h-1})2^{-h[(2\ell-k+1)+(d+3)/(d-1)]}
	&\leq \sum_{h=h_1}^{\infty}2^{-h[(2\ell-k+1)+(d+3)/(d-1)]} \\
	&\ll n^{-(2\ell-k+1)}n^{-(d+3)/(d-1)}.
	\end{align*}
	On the other hand, let $i=h_1-h$. Then, we can perform the following estimate, namely,
	\begin{align*}
	\sum_{h=h_0}^{h_1-1}&\exp(-c_2(n-2\ell+k)2^{-h-1})2^{-h[(2\ell-k+1)+(d+3)/(d-1)]} \\
	&\leq\sum_{i=1}^{h_1-h_0}\exp(-c_2(n-2\ell+k)2^{-h_1+i-1})2^{-(h_1-i)[(2\ell-k+1)+(d+3)/(d-1)]} \\
	&\ll \sum_{i=1}^{h_1-h_0}\exp(-c_2(n-2\ell+k)2^{-h_1+i-1})n^{-(2\ell-k+1)}n^{-(d+3)/(d-1)}2^{i[(2\ell-k+1)+(d+3)/(d-1)]}\\
	&\ll n^{-(2\ell-k+1)}n^{-(d+3)/(d-1)}\sum_{i=1}^{\infty}\exp(-c_2 2^i)2^{i[(2\ell-k+1)+(d+3)/(d-1)]} \\
	&\ll n^{-(2\ell-k+1)}n^{-(d+3)/(d-1)}\sum_{j=1}^{\infty}\exp(-c_2 j)j^{5d} \\
	&\ll n^{-(2\ell-k+1)}n^{-(d+3)/(d-1)}.
	\end{align*}
	As a consequence, it holds
	\begin{equation*}
	\Sigma_k\ll n^{2\ell-k+1}\int_{G(d,\ell)}n^{-(2\ell-k+1)}n^{-(d+3)/(d-1)}\nu_{\ell}(\dint A)\ll n^{-(d+3)/(d-1)}.
	\end{equation*}
	Finally, the upper bounds are proven by summing up all $\Sigma_k$, $k=0,\dots,\ell$, in \Cref{eq:sigmak}.
	
	\bigskip
	
	In order to extend the proof to the case of a convex body \(K\in\mathcal{K}^2_+\), we follow the ideas presented in \cite[Section 6]{Barany2010}. By the compactness of \(\partial K\), there exist \(\gamma>0\) and \(\Gamma>0\), the global upper and the global lower bound on the principal curvatures of \(\partial K\), respectively. In our setting, all projected images of \(\partial K\) also have a boundary with the same properties as \(\partial K\), see for example \cite[Remark 5]{Schneider2014}. Without loss of generality we can choose \(\gamma\) and \(\Gamma\) to be also a bound on the principal curvatures of the boundaries of all \(\ell\text{-dimensional} \) projections of \(K\). Hence, one can locally approximate \(\partial K\) with affine images of balls and the volume of a \(\ell\text{-dimensional}\) cap with height \(t>0\) has order \(t^{\frac{\ell+1}{2}}\). Finally, \cite[Equation (27)]{Barany2010} ensures that \Cref{eq:boundbeta} still holds. 
\end{proof}

In the fashion of \cite[Section 7]{Barany2010}, we derive strong laws of large numbers from the upper variance bounds together with the following result of \cite{Reitzner2002}.

\begin{proposition}\textup{\cite[Theorem 1]{Reitzner2002}}
	Let \(K\in \mathcal{K}^2_+\) and choose $n$ random points on $\partial K$ independently 
	and according to the probability distribution $\mathcal{H}^{d-1}$. Then, there exist positive constants 
	\(c_{d,\ell,K}\) depending on \(d,\ell\) and the principal curvatures of \(K\) such that 
	\begin{equation}
	\label{eq:meanvolumes}
		\lim_{n\to\infty} \bigl(V_\ell(K)-\E[V_\ell(K_n)]\bigr)\cdot n^{\frac{2}{d-1}}=c_{d,\ell,K}, \quad \ell\in\{1,\dots,d\}.
	\end{equation}
\end{proposition}
	For the sake of brevity, the explicit expression of \(c_{d,\ell,K}\) is omitted here. It can be found in \cite[Equation (2)]{Reitzner2002}.
	
\begin{proof}[Proof of \Cref{thm:strongnumbers}]
	Let $\ell\in\{1,\dots,d\}$. Chebyshev's inequality and the variance upper bound yield
	\begin{equation*}
		\P\bigl( \bigl|V_\ell(K)-V_\ell(K_n)-\E\bigl[V_\ell(K)-V_\ell(K_n)\bigr]\bigr|\cdot n^{\frac{2}{d-1}}\geq\varepsilon\bigr)\le \varepsilon^{-2}n^{\frac{4}{d-1}}\Var[V_\ell(K_n)]\ll n^{-1}.
	\end{equation*}
	Select now the subsequence of indices \(n_k=k^2\). Then, it follows
	\begin{equation*}
	\sum_{k=1}^\infty 		\P\bigl( \bigl|V_\ell(K)-V_\ell(K_{n_k})-\E\bigl[V_\ell(K)-V_\ell(K_{n_k})\bigr]\bigr|\cdot n_k^{\frac{2}{d-1}}\geq\varepsilon\bigr)\ll 	\sum_{k=1}^\infty k^{-2}<\infty.
	\end{equation*}
	Applying the Borel-Cantelli Lemma together with \Cref{eq:meanvolumes}, we obtain that
	\begin{equation*}
		\lim_{k\to\infty} \bigl(V_\ell(K)-V_\ell(K_{n_k})\bigr)\cdot n_k^{\frac{2}{d-1}}=c_{d,\ell,K}
	\end{equation*}
	holds with probability 1. Note that \(V_\ell(K)-V_\ell(K_{n})\) is a decreasing and positive sequence. Therefore, this gives
		\begin{equation*}
 \bigl(V_\ell(K)-V_\ell(K_{n_{k}})\bigr)\cdot n_{k-1}^{\frac{2}{d-1}}\le \bigl(V_\ell(K)-V_\ell(K_{n})\bigr)\cdot n^{\frac{2}{d-1}}\le	\bigl(V_\ell(K)-V_\ell(K_{n_{k-1}})\bigr)\cdot n_k^{\frac{2}{d-1}}, 
 \end{equation*}
whenever \(n_{k-1}\le n\le n_k\). Taking the limit as \(k\to\infty\), \(n_{k-1}/n_k\to 1\), which allows us to conclude that the desired limit is reached by the whole sequence with probability \(1\).
\end{proof}
\section{Central limit theorems}
\label{Sec:centrallimitheorems}
In this last section, we prove the central limit theorems. In contrast to \cite{TTW2017}, where floating bodies were used, here 
we work with surface bodies as it was already done in \cite{Thale2017} for the case of the volume. In addition to that, we make use of the normal approximation bound of \Cref{prop:normalbound}.
\begin{proof}[Proof of \Cref{thm:CLT}]
	First, we prove the central limit theorems for $K=B^d$. For this reason, let us introduce the two events $B_1$ and $B_2$. The event that the random polytope $[X_2,\dots,X_n]$ contains the surface body $K(s\geq\tau_n)$ is denoted by $B_1$. Due to the definition of $B_1$, it follows by \Cref{lemma:surface-body} that
	\begin{equation*}
	\P(B_1^c)\leq c_1 n^{-\alpha},
	\end{equation*}
	where $c_1\in(0,\infty)$ is independent of $n$. We denote by $B_2$ the event that the random polytope $\bigcap_{W\in\{Y,Y',Z,Z'\}}[W_4,\ldots,W_n]$ contains the surface body $K(s\geq\tau_n)$, where $Y,Y',Z,Z'$ are recombinations of the random vector $X=(X_1,\ldots,X_n)$. By taking the union bound, we obtain
	\begin{equation*}
	\P(B_2^c)\leq c_2 n^{-\alpha},
	\end{equation*}
	where $c_2\in(0,\infty)$ is again independent of $n$. Next, for any \(\ell\in\{1,\ldots,d\}\), we apply the bound in \Cref{prop:normalbound} to the random variables
	\begin{equation*}
	W=f(X_1,\ldots,X_n):=V_{\ell}([X_1,\ldots,X_n])-\E\bigl[V_{\ell}(K_n)\bigr].
	\end{equation*}
	Note that $D_iW=D_iV_{\ell}(K_n)$ and $D_{i_1,i_2}W=D_{i_1,i_2}V_{\ell}(K_n)$ for $i,i_1,i_2\in\{1,\ldots,n\}$. Conditioned on the event $B_1$, we obtain from \eqref{eq.Kubota},
	\begin{equation}
	\label{eq:D1}
	D_1V_{\ell}(K_n)= \binom{d}{\ell}\frac{\kappa_d}{\kappa_\ell\kappa_{d-\ell}}\int_{G(d,\ell)}\mathrm{ vol}_{\ell}\bigl((K_n|L)\!\setminus\!([X_2,\dots,X_n]|L)\bigr)\,\nu_{\ell}(\textup{d} L).
	\end{equation}
	We now define a full-dimensional cap \(C\) in such a way that \(K_n\setminus [X_2,\ldots,X_n]\) is contained in \(C\). Consider now the visibility region \(	\mathrm{Vis}_{X_1}(\tau_n)\) of \(X_1\). By definition of the event \(B_1\), the surface body and by Lemma \ref{relation-of-caps}, the diameter of this visibility region is at most \(c_3 \tau_n^{1/(d-1)}\), where \(c_3>0\). We now indicate with \(D(X_1,c_3 \tau_n^{1/(d-1)})\) the points on \(\partial K\) with distance at most \(c_3\tau_n^{1/(d-1)}\) from \(X_1\). Then, \(C\coloneqq\conv\bigl\{D(X_1,c_3\tau_n^{1/(d-1)})\bigr\}\) is a spherical cap and it follows from \Cref{lemma-Vu} that \(C\) has volume of order at most \(\tau_n^{(d+1)/(d-1)}\). We call \(\alpha\) the central angle of \(C\). For any subspace \(L\in G(d,\ell)\), it holds that \((K_n|L)\setminus([X_2,\ldots,X_n]|L)\subseteq(C|L)\). We obtain \(\vol_{\ell}(C|L)\ll \tau_n^{(\ell+1)/(d-1)}\). Indeed, the height of \(C|L\) has the same order as the height of \(C\), namely \(\tau_n^{2/(d-1)}\), while the order of its base changes from \(((\tau_n)^{1/(d-1)})^{d-1}\) to  \(((\tau_n)^{1/(d-1)})^{\ell-1}\), since the dimension of \(L\) is \(\ell\). By construction of \(C\), it now follows that if \(\sphericalangle(X_1,L)\), the angle between \(X_1\) and \(L\), is too wide compared to \(\alpha\), then \(C|L\subseteq K_n|L\), for sufficiently large \(n\). Whenever this occurs, it also holds in particular that \((K_n\setminus[X_2,\ldots,X_n])|L\subseteq K_n|L\), i.e., \(K_n|L=[X_2,\ldots,X_n]|L\). In fact, one can check that the integrand in \eqref{eq:D1} can only be non-zero if \(\sphericalangle(X_1,L)\ll\alpha\). Therefore, we can restrict the integration to the set \(\{L\in G(d,\ell):\sphericalangle(X_1,L)\ll\alpha\} \). Moreover, it holds that \(\alpha\ll V_d(C)^{1/(d+1)}\), see e.g. \cite[Equation (21)]{Barany2010}. 
%	\begin{framed}
%		(It is wrong at the moment but we can write it similarly as in our first paper.) 
%		
%		According to the definition of $B_1$ and Lemma \ref{relation-of-caps}, $(K_n|L)\setminus ([X_2,\dots,X_n]|L)$ is contained in a cap $C$ with
%		\begin{equation*}
%		V_d(C)\ll \tau_n^{d+1\over d-1}.
%		\end{equation*} 
%		The height of $C$ is of order $\tau_n^{2\over d-1}$ and the $(d-1)$-dimensional base is of order $\tau_n^{d-1\over d-1}=\tau_n$. We note that the height of $C|L$ keeps the order of $C$, while the order of its base changes to $\tau_n^{\ell-1\over d-1}$. Therefore, one has
%		\begin{equation*}
%		\vol_{\ell}(C|L)\ll \tau_n^{\ell+1\over d-1}.
%		\end{equation*}
%		Let $z$ be the centre of $C$ on the boundary of $B^d$. Moreover, we denote by $\alpha$ the central angle of $C$. Then, $(K_n|L)\setminus ([X_2,\dots,X_n]|L)$ is always empty whenever $\sphericalangle(z,L)$, the angle between $z$ and $L$, is bigger than $2\alpha$. As a consequence, we restrict the integration to $\{L\in G(d,\ell):\sphericalangle(z,L)\leq 2\alpha\}$. Moreover, we have $\alpha\ll V_d(C)^{1\over d+1}$. 
%	\end{framed}
	According to \Cref{lemma:measure}, this gives
		\begin{equation*}
		\nu_{\ell}\bigl(\bigl\{L\in G(d,\ell):\sphericalangle(X_1,L)\ll V_d(C)^{\frac{1}{d+1}}\bigr\}\bigr)\ll \tau_n^{\frac{d-\ell}{d-1}}.
		\end{equation*}
	Putting everything together, we see that
	\begin{align}\label{eq:Bound1stOrderDifference}
	D_1V_{\ell}(K_n)\ll \tau_n^{\frac{\ell+1}{d-1}}\cdot \tau_n^{\frac{d-\ell}{ d-1}}\ll\Bigl(\frac{\log{n}}{n}\Bigr)^{\frac{d+1}{d-1}}.
	\end{align}
	On the complement $B_1^c$ of $B_1$ we use the trivial estimate $D_1V_{\ell}(K_n)\leq V_{\ell}(K)$. Since $\P(B_1^c)\ll n^{-\alpha}$, we obtain
	\begin{align*}
	\E[(D_1V_{\ell}(K_n))^p] &=\E[(D_1V_{\ell}(K_n))^p\,\indicator_{B_1}]+\E[(D_1V_{\ell}(K_n))^p\,\indicator_{B_1^c}]\\
	&\ll \Bigl(\frac{\log{n}}{n}\Bigr)^{p \frac{d+1}{d-1}},
	\end{align*} 
	for all $p\geq1$. As a consequence, we can bound the terms in the normal approximation bound which involve $\gamma_3$ and $\gamma_4$. Thus,
	\begin{align*}
	\frac{\sqrt{n}}{\Var[V_{\ell}(K_n)]}\sqrt{\gamma_3} &\ll\frac{\sqrt{n}}{n^{-\frac{d+3}{d-1}}} \Bigl(\frac{\log{n}}{n}\Bigr)^{2 \frac{d+1}{d-1}}= n^{-\frac{1}{2}}(\log{n})^{2+\frac{4}{d-1}}\,,\\
	\frac{n}{(\Var[V_{\ell}(K_n)])^{\frac{3}{2}}}\,\gamma_4 & \ll \frac{n}{n^{-\frac{3}{2}\frac{d+3}{d-1}}}\Bigl(\frac{\log{n}}{n}\Bigr)^{3 \frac{d+1}{d-1}}= n^{-\frac{1}{2}}(\log{n})^{3+\frac{6}{d-1}}\,.
	\end{align*}
	By using the Cauchy-Schwarz inequality, we can estimate $\gamma_5$ as well. Namely,
	\begin{align*}
	\gamma_5\leq \sqrt{\Var[V_{\ell}(K_n)]}\, \sup_{A\subseteq\{1,\dots,n\}}\sqrt{\E\bigl[|D_1f(X^{A})|\bigr]^6}\ll n^{-{1\over2}{d+3\over d-1}} \Bigl(\frac{\log{n}}{n}\Bigr)^{3 \frac{d+1}{d-1}}.
	\end{align*}
	Thus, we obtain
	\begin{align*}
	{n\over(\Var [V_{\ell}(K_n)])^2}\,\gamma_5\ll {n\over n^{-2{d+3\over d-1} }}\,n^{-{1\over2}{d+3\over d-1}} \Bigl(\frac{\log{n}}{n}\Bigr)^{3 \frac{d+1}{d-1}}= n^{-{1\over 2}}(\log{n})^{3+{6\over d-1}}.
	\end{align*}
	In the next step, we consider the terms involving the second difference operator.
	%and
	%\begin{equation*}
	%\Delta(\tau_n):=\sup_{z\in\partial K}V_n(\Vis_z(\tau_n))\ll \tau_n^{n+1\over n-1}=\left(\frac{\log{N}}{N}\right)^{n+1\over n-1}.
	%\end{equation*}
	On the event $B_2$ it may be concluded from \eqref{eq:Bound1stOrderDifference} that $D_if(V)^2\ll (\log{n}/n)^{2{d+1\over d-1}}$ for all $i\in\{1,2,3\}$ and $V\in\{Z,Z'\}$. Moreover, we note that on $B_2$ the following inclusions hold
	\begin{equation*}
	\{D_{1,2}f(Y)\neq 0\}\subseteq\{\mathrm{Vis}_{Y_1}(\tau_n)\cap\mathrm{Vis}_{Y_2}(\tau_n)\neq\emptyset\}\subseteq \biggl\{Y_2\in\!\!\!\!\!\bigcup_{x\in\mathrm{Vis}_{Y_1}(\tau_n)}\!\!\!\!\mathrm{Vis}_x(\tau_n)\biggr\}.
	\end{equation*}
	The same applies to $D_{1,3}f(Y')$. Thus,
	\begin{align*}
	\E\bigl[\indicator\{D_{1,2}f(Y)\neq 0\}\indicator_{B_2}\bigr]\leq \sup_{z\in\partial K}\P\biggl(Y_2\in\!\!\!\!\!\bigcup_{x\in\mathrm{Vis}_{z}(\tau_n)} \!\!\!\!\!\mathrm{Vis}_x(\tau_n)\biggr).
	\end{align*}
%	Since the diameter of the previous union is of order $\tau_n^{1\over d-1}$, it follows from \cite[Lemma~6.2]{Vu2005} that 
%	\begin{equation*}
%	V_d\biggl(\bigcup_{x\in\mathrm{Vis}_{z}(\tau_n)}\!\!\!\!\mathrm{Vis}_x(\tau_n)\biggr)\ll \tau_n^{d+1\over d-1}, \quad z\in\partial K.
%	\end{equation*}
	 We note that the diameter of the previous union is at most \(c_4 \tau_n^{1/(d-1)}\), where \(c_4>0\). 	As before, we define the spherical cap \(C'\coloneqq\conv\{D(z,c_4\tau_n^{1/(d-1)}) \}\). It follows from \Cref{lemma-Vu} that \(C'\) has volume of order at most \(\tau_n^{(d+1)/(d-1)}\).
		We obtain
		\begin{equation*}
		\begin{split}
		\sup_{z\in\partial K}\P\biggl(Y_2\in\!\!\!\!\bigcup_{x\in\mathrm{Vis}_{z}(\tau_n)} \!\!\!\!\mathrm{Vis}_x(\tau_n)\biggr)&=\sup_{z\in\partial K}\mathcal{H}^{d-1}\biggl(\biggl(\bigcup_{x\in\mathrm{Vis}_{z}(\tau_n)}\!\!\!\! \mathrm{Vis}_x(\tau_n)\biggr)\cap \partial K\biggr)\\
		&\le \sup_{z\in\partial K}\mathcal{H}^{d-1}\bigl(C'\cap \partial K\bigr)\\
		&\ll \tau_n,
			\end{split}
		\end{equation*}
where for the last inequality we have used \Cref{relation-of-caps}.
	On the event $B_2^c$ we use the trivial estimate $V_{\ell}(K)$ for all difference operators and estimate all indicators by one. Since $\P(B_2^c)\ll n^{-\alpha}$, we obtain
	\begin{equation*}
	\gamma_2\ll \Bigl(\frac{\log{n}}{n}\Bigr)^{1+4 \frac{d+1}{d-1}}.
	\end{equation*}
	Analogously, we can bound $\gamma_1$. Indeed, suppose that $Y_1=Y'_1$ (by independence, $Y_1\neq Y'_1$ gives a smaller order), then
	\begin{equation*}
	\{D_{1,2}f(Y)\neq0\}\cap\{D_{1,3}f(Y')\neq 0\}\subseteq\biggl\{\{Y_2,Y_3'\}\subseteq\!\!\!\!\bigcup_{x\in\mathrm{Vis}_{Y_1}(\tau_n)} \!\!\!\!\mathrm{Vis}_x(\tau_n)\biggr\}
	\end{equation*}
	and we obtain
	\begin{equation*}
	\E\bigl[\indicator\{D_{1,2}f(Y)\neq 0\}\indicator\{D_{1,3}f(Y')\neq 0\}\bigr]\ll \Bigl(\frac{\log{n}}{n}\Bigr)^2.
	\end{equation*}
	Thus,
	\begin{equation*}
	\gamma_1\ll\Bigl(\frac{\log{n}}{n}\Bigr)^{2+4 \frac{d+1}{d-1}}.
	\end{equation*}
	Finally,
	\begin{align*}
	{\sqrt{n}\over\Var [V_{\ell}(K_n)]}\,\sqrt{n^2\gamma_1}& \ll {\sqrt{n}\over n^{-{d+3\over d-1}}}\,\sqrt{n^2\Bigl(\frac{\log{n}}{n}\Bigr)^{2+4 \frac{d+1}{d-1}}}=n^{-\frac{1}{2}}(\log n)^{3+\frac{4}{d-1}},\\
	\frac{\sqrt{n}}{\Var [V_{\ell}(K_n)]}\,\sqrt{n\gamma_2} &\ll {\sqrt{n}\over n^{-{d+3\over d-1}}}\,\sqrt{n \Bigl(\frac{\log{n}}{n}\Bigr)^{1+4 \frac{d+1}{d-1}}}=n^{-\frac{1}{2}}(\log n)^{\frac{5}{2}+\frac{4}{d-1}}.
	\end{align*}
	Considering all the estimates together, we obtain by \Cref{prop:normalbound}
	\begin{equation*}
	\begin{split}
	d_K\bigl(W_\ell(K_n),N\bigr)&\ll n^{-\frac{1}{2}}\bigl((\log n)^{3+\frac{4}{d-1}}+(\log n)^{\frac{5}{2}+\frac{4}{d-1}}\\
	&\qquad\qquad\qquad+(\log n)^{2+\frac{4}{d-1}}+(\log n)^{3+\frac{6}{d-1}}+(\log n)^{3+\frac{6}{d-1}}\bigr)\\
	&\ll n^{-\frac{1}{2}}(\log n)^{3+\frac{6}{d-1}}.
	\end{split}
	\end{equation*}
For the case of a generic \(K\in\mathcal{K}^2_+\) we argue as at the end of the proof of the upper bounds of \Cref{thm:variancebounds}. Because of the global bounds on the principal curvatures and the local approximation of \(\partial K\) with affine images of balls, the construction of \(C\) and the relations regarding its volume, its central angle and the subspaces \(L\) which ensure \(C|L\subseteq K_n|L\) are not afflicted. In particular, the asymptotic bounds \(\vol_{\ell}(C|L)\ll\tau_n^{{(\ell+1)}/{(d-1)}}\), \(\alpha\ll V_d(C)^{1/(d+1)}\ll\tau_n
^{1/(d-1)} \) and \(\sphericalangle(X_1,L)\ll\alpha \) stated above still hold, with the difference that the implicit constants depend on \(\gamma \) and \(\Gamma\), the bounds on the principal curvatures of \(\partial K\). Then, the proof can be completed like in the case of the ball.
\end{proof}
	\large\textbf{Acknowledgement} %Uncomment for arxiv
	
	\bigskip 
	
	\normalsize We would like to thank Christoph Th\"ale for helpful discussions and insights.

\bibliographystyle{abbrv}
\bibliography{references}

\end{document}